\documentclass[11pt,reqno]{amsart}
\usepackage[T1]{fontenc}
\usepackage[utf8]{inputenc}
\usepackage{lmodern}
\usepackage{amssymb,mathrsfs}
\usepackage{geometry}
\usepackage{microtype}
\usepackage{tikz}
\usepackage{cite}
\usepackage[colorlinks,linkcolor=blue,citecolor=red,urlcolor=blue]{hyperref}
\usepackage{indentfirst}
\theoremstyle{plain}
\newtheorem{theorem}{Theorem}[section]
\newtheorem{proposition}[theorem]{Proposition}
\newtheorem{lemma}[theorem]{Lemma}

\theoremstyle{remark}
\newtheorem{remark}[theorem]{Remark}
\newtheorem{example}[theorem]{Example}
\newcommand{\PP}{\mathbb P}
\newcommand{\EE}{\mathbb E}
\newcommand{\ZZ}{\mathbb Z}

\newcommand{\RR}{\mathbb R}
\newcommand{\FF}{\mathcal F}
\newcommand{\one}{\mathbf 1}
\newcommand{\cE}{\mathcal E}

\title{Recurrence and range of the balanced excited random walk M(2,1,2)}

\author[Shuo Qin]{Shuo Qin}
\address[Shuo Qin]{Beijing Institute of Mathematical Sciences and Applications, and Yau Mathematical Sciences Center, Tsinghua University}
\email{qinshuo@bimsa.cn}

\date{}
\begin{document}
\begin{abstract}
We prove that the planar balanced excited random walk $M(2,1,2)$ is recurrent. This walk takes a horizontal simple random walk step on its first departure from each vertex and a planar simple random walk step on every later departure. Moreover, the number of distinct vertices visited before time $n$, multiplied by $(\log n)/n$, converges to $\pi$ almost surely and in every $L^p$, $1\le p<\infty$, the same limit as for the planar simple random walk. More generally, we prove recurrence of balanced excited random walks in spatially inhomogeneous cookie environments whenever the total positive and negative cookie strengths at each vertex are bounded by constants $A, B$ with $A+B<1+1/(2\pi +1)$.
\end{abstract}
\maketitle

\section{Introduction}

Consider a walk $S=(S_n)_{n\ge0}$ on $\ZZ^2$ started at the origin. On its first
departure from each vertex, it steps to the left or to the right, with
equal probabilities. On every later departure from that vertex,
it chooses uniformly among the four neighboring vertices. This
walk is denoted by $M(2,1,2)$. Write
\[
 R_n=\bigl|\{S_0,\ldots,S_{n-1}\}\bigr|,\qquad n\ge1,
\]
for the number of vertices visited before time $n$.

\begin{theorem}\label{thm:m212}
The walk $M(2,1,2)$ almost surely visits every vertex of $\ZZ^2$
infinitely often. Moreover,
\begin{equation}\label{eq:m212-range}
 \lim_{n\to\infty}\frac{\log n}{n}R_n=\pi
 \quad\text{a.s. and in }L^p,\quad 1\le p<\infty.
\end{equation}
\end{theorem}

The conditional mean of each step is zero. Nevertheless, a first
departure can increase the distance from the origin more effectively than a planar simple random walk step: a horizontal step from the vertical axis always increases that distance. Recurrence therefore
requires control of the cumulative effect of first departures.
The range asymptotic in \eqref{eq:m212-range} shows that their number has the same first-order asymptotic as the range of planar simple random walk \cite{MR0047272}. We give some estimates for the convergence rate in Proposition \ref{prop:range-limit}.

The notation $M(2,1,2)$ comes from the family $M(d,d_1,d_2)$ of balanced excited random walks. Let $e_1,\ldots,e_d$ be the standard coordinate vectors of $\RR^d$. For integers $d\ge2$ and $1\le d_1,d_2\le d$ with $d_1+d_2\ge d$, this walk $M(d,d_1,d_2)$ moves on $\ZZ^d$. On a first departure, it chooses uniformly among $\{\pm e_i:1\le i\le d_1\}$; on a later departure, it chooses uniformly among $\{\pm e_i:d-d_2+1\le i\le d\}$. The two sets of coordinates may thus overlap. The two laws coincide only when $d_1=d_2=d$, in which case the walk is the $d$-dimensional simple random walk.

Benjamini, Kozma and Schapira \cite{MR2788390} introduced the case $d_1+d_2=d$, and Camarena, Panizo and Ram\'irez
\cite{MR4237270} considered overlapping coordinates.
The word \emph{balanced} refers to the zero conditional mean of each step. This differs from the excited random walk of Benjamini and Wilson \cite{MR1987097}, whose first departures have a
nonzero directional drift and which is transient in dimensions at least two.

In the non-overlapping case $d_1+d_2=d$, projections onto $\ZZ^{d_1}$ and $\ZZ^{d_2}$ imply that the walk $M(d,d_1,d_2)$ is transient if $\max\{d_1, d_2\} \geq 3$. Benjamini, Kozma and Schapira proved that $M(4,2,2)$ is transient. They conjectured that $M(3,1,2)$ and $M(3,2,1)$ are transient. The transience of $M(3,1,2)$ was proved by Peres, Schapira and Sousi \cite{MR3531697}. In two dimensions, Angel, Holmes and Ram\'irez \cite{MR4597323} studied $M(2,1,1)$, and obtained upper and lower bounds for its range. In particular, \cite[Theorem 4.2]{MR4597323} shows that the expected range is $O(n/\sqrt{\log\log n})$. We note that with a suitable timing rule, the same theorem gives this upper bound for the expected range of $M(2,1,2)$. The recurrence of $M(2,1,1)$ remains conjectural. In the overlapping case  $d_1+d_2>d$, Camarena, Panizo and Ram\'irez  \cite{MR4237270} proved that $M(d,d_1,d_2)$ is transient for every $d\ge4$, except for $M(4,3,2)$. To our knowledge, Theorem \ref{thm:m212} gives the first recurrence result for a member of this family, except the simple random walk case.

\subsection{Cookie environments}

Our argument allows the excitation at a vertex to be spread over several visits. It also allows both horizontal and vertical excitation, subject to common bounds on their total strengths. For background on excited random walks and cookie environments, we refer to the survey of Kosygina and Zerner \cite{MR3097419}.

A cookie environment is a family $\eta=(\beta_{z,j}:z\in\ZZ^2,j\ge1)$
with $\beta_{z,j}\in[-1,1]$. Given $\eta$ and a starting point
$x\in\ZZ^2$, let $S=(S_n)_{n\ge0}$ be a nearest-neighbour walk on
$\ZZ^2$ with $S_0=x$, and write $\PP_x^\eta$ for its law. Set
$e_1=(1,0)$, $e_2=(0,1)$, $\FF_n=\sigma(S_0,\ldots,S_n)$, and
\[
 L_n(z)=\sum_{k<n}\one_{\{S_k=z\}},\qquad
 b_n=\beta_{S_n,L_n(S_n)+1}.
\]
Thus $L_n(z)$ counts the departures from $z$ strictly before time
$n$, and $b_n$ is the strength of the cookie used for the departure
from $S_n$ at time $n$. The walk is characterized by the conditional
step law
\begin{equation}\label{eq:cookie-law}
 \PP_x^\eta(S_{n+1}-S_n=e\mid\FF_n)
 =\begin{cases}
   (1+b_n)/4,&e=\pm e_1,\\
   (1-b_n)/4,&e=\pm e_2.
  \end{cases}
\end{equation}
Equivalently, on the $j$-th departure from $z$, each horizontal
direction has probability $(1+\beta_{z,j})/4$ and each vertical
direction has probability $(1-\beta_{z,j})/4$. In particular,
strength $\beta_{z,j}=0$ gives a planar simple random walk step,
strength $1$ gives a horizontal step, and strength $-1$ gives a
vertical step.

For $t\in\RR$, write $t_+=\max(t,0)$ and $t_-=\max(-t,0)$.
For $A,B\ge0$, let $\cE_{A,B}$ denote the environments satisfying
\begin{equation}\label{eq:budgets}
 \sum_{j\ge1}(\beta_{z,j})_+\le A,\qquad
 \sum_{j\ge1}(\beta_{z,j})_-\le B
 \quad\text{for every }z\in\ZZ^2.
\end{equation}

\begin{theorem}\label{thm:main}
Suppose that $A,B\ge0$ and
\[
 A+B<\kappa_*:=1+\frac1{2\pi+1}.
\]
For every
$\eta\in\cE_{A,B}$ and every starting point $x\in\ZZ^2$, the
walk defined by \eqref{eq:cookie-law} almost surely visits every
vertex infinitely often.
\end{theorem}

The recurrence assertion of Theorem \ref{thm:m212} is the case
$\beta_{z,1}=1$ and $\beta_{z,j}=0$ for $j\ge2$. Other examples
include nonnegative stacks with total strength at most one, such
as $\beta_{z,j}=2^{-j}$. Positive and negative cookies may occur
in any order, including infinitely often. Here we treat deterministic environments (stacks). For a random environment independent of the randomness used to choose the steps, the theorem applies conditionally on every realization satisfying \eqref{eq:budgets} with $A+B<\kappa_*$.

The same constants $A,B$ must apply at every vertex. The condition
$\sum_j|\beta_{z,j}|\le1$ separately at each vertex does not imply
common bounds $A,B$ with $A+B<\kappa_*$: a cookie of strength $1$ at
one vertex and a cookie of strength $-1$ at another force $A\ge1$
and $B\ge1$. Such environments are not covered by
Theorem \ref{thm:main}.

We shall use the following quantitative estimate, which holds under
the weaker condition $AB<1$. Define the excitation consumed before time $n$ by
\begin{equation}
    \label{defXingen}
     \Xi_n=\sum_{k<n}|b_k|.
\end{equation}

\begin{proposition}\label{prop:occupation}
For every $A,B\ge0$ with $AB<1$ and every $n\ge1$,
\begin{equation}\label{eq:main-occupation}
 \sup_{\eta\in\cE_{A,B}}\sup_{x\in\ZZ^2}
 \bigl(\EE_x^\eta\Xi_n+\EE_x^\eta R_n\bigr)
 \le C_{A,B}\frac{n}{\log(n+2)},
 \qquad C_{A,B}=C\frac{(1+A)(1+B)}{1-AB},
\end{equation}
where $C$ is universal. The same bound holds conditionally after
any almost surely finite stopping time $T$, with $\Xi_n$ and $R_n$
replaced by $\sum_{j<n}|b_{T+j}|$ and
$|\{S_T,\ldots,S_{T+n-1}\}|$, respectively.
\end{proposition}

\begin{remark}
For cookies of one sign, Proposition \ref{prop:occupation} allows
any finite total budget: if $B=0$, the condition $AB<1$ holds for
every finite $A$, and symmetrically when $A=0$.
\end{remark}

For $M(2,1,2)$ started with every vertex unvisited, $\Xi_n=R_n$.
For general stacks, these quantities may differ, and Proposition \ref{prop:occupation}
controls both. Its uniformity allows us to use the estimate in
successive blocks whose lengths depend on the current position.

\subsection{Proof strategy}

We first prove the uniform excitation bound in
Proposition \ref{prop:occupation}. Consider nonnegative cookies
with total strength at most one at each vertex. At time $n$, put
mass at each vertex equal to the excitation already consumed there, and evaluate its potential at $S_n$ using the ordinary potential kernel $a$ of planar simple random walk. Denote the resulting potential by $\Phi_n$: 
\begin{equation}
    \label{defPhin}
 \mu_n(z):=\sum_{j<n}b_j\one_{\{S_j=z\}},\qquad \Phi_n:=\sum_z\mu_n(z)a(S_n-z).
\end{equation}
We prove
\[
 \EE\Phi_n\le Cn,\qquad
 \Phi_n\ge\frac1\pi\Xi_n\log\Xi_n-C\Xi_n.
\]
Figure \ref{fig:consumed-potential} illustrates the construction in the special case of $M(2,1,2)$, where each first departure deposits one unit of mass. For general nonnegative cookies, the mass at a vertex is the total excitation already consumed there. 

The lower bound for $\Phi_n$ follows from a packing inequality \eqref{eq:packing}, which uses the logarithmic growth of $a$
and the fact that each vertex carries at most one unit of mass. For the upper bound,
we sum the conditional drift identity, see \eqref{eq:potential-increment}. Its anisotropic correction becomes
the symmetric quadratic form in \eqref{eq:quadratic-identity},
which can be controlled using Fourier transformation, see Lemma \ref{lem:anisotropic}. 
For cookies of both signs, separate potentials give the two
inequalities in \eqref{eq:two-sign-system}, and they yield the desired bound when $AB<1$. A potential generated by the whole visited set
then gives the range bound. Substituting the resulting excitation bound
back into the packing estimate gives the leading coefficient in
\eqref{eq:leading-excitation}. Section \ref{sec:potential} develops
the potential estimates, and Section \ref{sec:range} proves these
uniform bounds and their moment consequences.

For $M(2,1,2)$, the consumed measure is unit mass on the visited
set. In Section \ref{sec:asymptotic}, the same quadratic form gives
a martingale decomposition, and we show
\[
 \Phi_n=n+o(n),\qquad
 \Phi_n=\frac{R_n}{\pi}\log n+o(n),
\]
both almost surely and in $L^1$. The second comparison uses the
packing lower bound and a maximal inequality confining the path
to a ball of radius $\sqrt n\log n$. These estimates explain the
constant $\pi$ in \eqref{eq:m212-range}. The higher moment bounds
give convergence in every finite $L^p$.

Section \ref{sec:recurrence} proves recurrence using
$\psi(z)=\log\log|z|$ outside a fixed ball. Let $P$ be
the planar simple random walk operator and $D$ the difference
between the horizontal simple random walk operator and $P$.
Writing the conditional drift as $(P-I)\psi+bD\psi$ retains
the negative ordinary drift at every step. The cookie budgets give one bound on
the accumulated excitation drift, and the uniform excitation estimate
on local blocks gives another. Combining these bounds yields a probability
of hitting a fixed ball bounded away from zero after every finite history.
Section \ref{sec:extensions} gives two further examples, allowing
infinite absolute excitation or large budgets on a sparse lattice.

Throughout the paper, $|z|$ denotes Euclidean norm,
$B(x,r)=\{z\in\ZZ^2:|z-x|<r\}$, and $B_r=B(0,r)$. Constants
$c,C>0$ may change from line to line and are universal unless
indicated otherwise. Constants denoted by $C_{A,B}$ may depend
on $A,B$, but not on the environment or the starting point. Further
parameter dependence is specified where needed. We use $0\log0=0$.

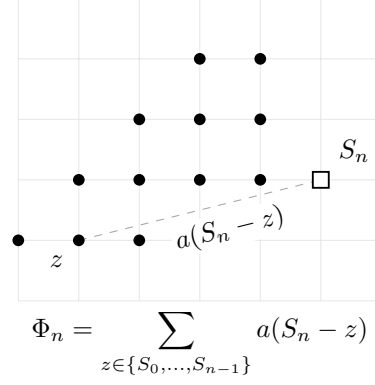
\begin{figure}[t]
\centering
\begin{tikzpicture}[x=0.8cm,y=0.8cm,font=\small]
  \draw[step=1,gray!20,very thin] (0,-1) grid (6,4);

  \coordinate (z) at (1,0);
  \coordinate (S) at (5,1);

  \draw[dashed,gray!70] (z) -- (S)
    node[pos=0.6,sloped,below=3pt,
         fill=white,text=black,inner sep=1.5pt]
    {$a(S_n-z)$};

  \foreach \p in {(0,0),(1,0),(2,0),(1,1),(2,1),(3,1),
                  (2,2),(3,2),(4,2),(3,3),(4,3),(4,1)}
    \fill \p circle (2.2pt);

  \node[below left=2pt] at (z) {$z$};

  \node[draw,fill=white,rectangle,line width=0.7pt,
        minimum size=6pt,inner sep=0pt] at (S) {};
  \node[above right=3pt] at (S) {$S_n$};

  \node at (3,-1.7)
    {$\displaystyle
      \Phi_n=\sum_{z\in\{S_0,\ldots,S_{n-1}\}}a(S_n-z)$};
\end{tikzpicture}
\caption{Unit masses deposited by first departures before time $n$
in $M(2,1,2)$. The potential is evaluated at $S_n$, shown as a
square at a newly visited vertex.}
\label{fig:consumed-potential}
\end{figure} 

\section{Potential kernel estimates}\label{sec:potential}

Let $P_h$ and $P_v$ denote the horizontal and vertical simple random
walk operators, and put
\[
 P=\frac{P_h+P_v}{2},\qquad D=P_h-P=\frac{P_h-P_v}{2}.
\]
The one-step operator with strength $b$ is $P+bD$. Note that each step is centered with unit length, and thus $(|S_n|^2-n)$ is a martingale. Let $a$ be the
potential kernel of planar simple random walk, normalized by
\[
 a(0)=0,\qquad (P-I)a=\one_{\{0\}}.
\]
It is even (i.e., $a(z)=a(-z)$) and nonnegative, $a(\pm e_i)=1$, and
\begin{equation}\label{eq:potential-asymptotic}
 a(z)=\frac2\pi\log|z|+c_a+O(|z|^{-2})
 \qquad (|z|\longrightarrow\infty).
\end{equation}
These standard facts, including the Fourier representation used
below, can be found in Chapter 4 of Lawler and Limic
\cite{MR2677157}, see in particular \cite[Proposition 4.4.3, Theorem 4.4.4]{MR2677157}.

For any $x,z \in \ZZ^2$ and $j\ge0$, the relation $(P-I)a=\one_{\{0\}}$ gives
\begin{equation}
\label{EcondincaSz}
    \EE_x^{\eta} [a(S_{j+1}-z) -a(S_j-z) \mid \FF_j]=(P-I+b_j D)a(S_j-z)=\mathbf{1}_{\{S_j=z\}} +b_jDa(S_j-z).
\end{equation}
We need two estimates for the anisotropic correction $d=Da$.
The first bounds its quadratic form. The second bounds the sum of
either sign over any set of a given size. The coefficient in the
second estimate is needed when both signs of cookies are present.

\begin{lemma}\label{lem:anisotropic}
The function $d=Da$ is even and $d(0)=0$. For every finitely
supported real function $f$ on $\ZZ^2$,
\begin{equation}\label{eq:quadratic-bound}
 \left|\sum_{z,w}f(z)f(w)d(z-w)\right|
 \le\sum_z f(z)^2.
\end{equation}
Moreover, for $z=(x,y)\ne0$,
\begin{equation}\label{eq:d-asymptotic}
 d(z)=\frac{y^2-x^2}{\pi|z|^4}+O(|z|^{-4}).
\end{equation}
Write $d_+=\max(d,0)$ and $d_-=\max(-d,0)$ for the positive and
negative parts of $d$. There is a constant $C$ such that every
finite set $E$ with $|E|\le N$ satisfies, for every $x\in\ZZ^2$,
\begin{equation}\label{eq:signed-rows}
 \sum_{z\in E}d_+(x-z)\le\frac1\pi\log(N+2)+C,
 \qquad
 \sum_{z\in E}d_-(x-z)\le\frac1\pi\log(N+2)+C.
\end{equation}
In particular, $\sum_{z\in E}|d(x-z)|\le C\log(N+2)$.
\end{lemma}

\begin{proof}
The Fourier representation of $a$ is
\[
 a(z)=\int_{[-\pi,\pi]^2}
 \frac{1-\cos(\xi\cdot z)}{1-\phi(\xi)}\,
 \frac{d\xi}{(2\pi)^2},
 \qquad \phi(\xi)=\frac{\cos\xi_1+\cos\xi_2}{2}.
\]
Here $\phi$ is the characteristic function of a planar simple
random walk step. Since
\[
 D\cos(\xi\cdot z)
 =\frac{\cos\xi_1-\cos\xi_2}{2}\cos(\xi\cdot z),
\]
applying $D$ to the representation of $a$ gives
\[
 d(z)=\int_{[-\pi,\pi]^2}m(\xi)\cos(\xi\cdot z)
       \frac{d\xi}{(2\pi)^2},
\]
where
\begin{equation}\label{eq:multiplier}
 m(\xi)=\frac{\cos\xi_2-\cos\xi_1}
 {2-\cos\xi_1-\cos\xi_2}.
\end{equation}
The value at zero is immaterial. The inequality
$|\cos\xi_1-\cos\xi_2|\le2-\cos\xi_1-\cos\xi_2$ gives
$|m|\le1$. If $\widehat f(\xi)=\sum_z f(z)e^{i\xi\cdot z}$,
then
\[
 \sum_{z,w}f(z)f(w)d(z-w)
 =\int_{[-\pi,\pi]^2}m(\xi)|\widehat f(\xi)|^2
 \frac{d\xi}{(2\pi)^2}.
\]
Parseval's identity proves \eqref{eq:quadratic-bound}. Note that since $f$ has a finite support, this calculation requires no absolute summability of $d$. Evenness follows from \eqref{eq:multiplier}.
Interchanging the two coordinates changes the sign of $d(z)$
and gives $d(0)=0$.

To prove \eqref{eq:d-asymptotic}, we use the more precise expansion
of Kozma and Schreiber \cite[Equation (4)]{MR2041826}:
\[
 a(z)=\frac2\pi\log|z|+c_a
      +h(z)+\rho(z),\qquad z=(x,y) \ne 0,
\]
where
\[
 h(z)=-\frac{\operatorname{Re}\bigl((x+iy)^4\bigr)}
                  {6\pi|z|^6},\qquad z=(x,y)\ne0,
\]
and $\rho(z)=O(|z|^{-4})$. The function $h$ is smooth away from the
origin and satisfies $h(tz)=t^{-2}h(z)$ for $t>0$. Its second
derivatives are therefore homogeneous of degree $-4$, so
\[
 |\partial_i^2h(z)|\le C|z|^{-4},\qquad i=1,2.
\]
For sufficiently large $|z|$, Taylor's formula gives
\[
 h(z+e_i)-2h(z)+h(z-e_i)=\int_{-1}^{1}(1-|t|) \partial_i^2h(z+te_i)\,dt =O(|z|^{-4}),
\]
since $|z+te_i|\ge |z|/2$ for $|t|\le1$.
As $D$ is one quarter of the difference of these two central
second differences, this proves $Dh(z)=O(|z|^{-4})$.
For the remainder, the triangle inequality alone gives
\[
 |D\rho(z)|
 \le\frac14\sum_{i=1}^{2}
       \bigl(|\rho(z+e_i)|+|\rho(z-e_i)|\bigr)
 \le C|z|^{-4}.
\]
Thus no derivative estimate on $\rho$ is needed. Taylor expansion of the logarithmic term therefore yields \eqref{eq:d-asymptotic}:
\[
 d(z)=Da(z)=\frac1{2\pi}(\partial_x^2-\partial_y^2)\log|z|
        +O(|z|^{-4})
      =\frac{y^2-x^2}{\pi|z|^4}+O(|z|^{-4}).
\]

In polar coordinates, the leading term of $d(z)$ is
$-\cos(2\theta)/(\pi r^2)$. The angular integral of each of $(-\cos (2\theta))_+$ and $(-\cos (2\theta))_{-}$ equals $2$. Hence
\begin{equation}\label{eq:disk-rows}
 \sum_{|z|\le r}d_\pm(z)=\frac2\pi\log r+O(1),
 \qquad r\ge2.
\end{equation}
For the comparison with a polar integral, the remainder in
\eqref{eq:d-asymptotic} is summable. The positive and negative
parts of $(y^2-x^2)/(\pi|z|^4)$ are locally Lipschitz, with Lipschitz
bound $C|z|^{-3}$ on unit squares away from zero. The sum of the
errors over these squares is therefore bounded. This justifies
the error term in \eqref{eq:disk-rows}.

Split $E-x$ at radius $\sqrt N$. Equation \eqref{eq:disk-rows}
bounds the inner contribution by $\pi^{-1}\log N+C$. Outside
that disk, \eqref{eq:d-asymptotic} bounds each term by $C/N$;
there are at most $N$ terms. Enlarging $C$ for bounded $N$ proves
\eqref{eq:signed-rows}. Adding the two inequalities gives the
last assertion.
\end{proof}

The following packing estimate states that a measure with large total mass must generate a large potential, provided that no vertex carries more than one unit.

\begin{lemma}\label{lem:packing}
Let $\mu$ be a finitely supported measure on $\ZZ^2$ with
$0\le\mu(z)\le1$, and let $M=\sum_z\mu(z)$. Then, for every
$x\in\ZZ^2$,
\begin{equation}\label{eq:packing}
 \sum_z\mu(z)a(x-z)\ge\frac1\pi M\log M-CM.
\end{equation}
More generally, if $0\le\mu(z)\le K$ for some $K>0$, then
\begin{equation}\label{eq:packing-scaled}
 \sum_z\mu(z)a(x-z)\ge\frac M\pi\log\frac M K-CM.
\end{equation}
\end{lemma}

\begin{proof}
Order the lattice points $z_k$ by increasing distance from $x$.
Counting lattice points in disks gives $1+|z_k-x|^2\ge ck$.
Equation \eqref{eq:potential-asymptotic} then implies
\[
 a(x-z_k)\ge\frac1\pi\log k-C.
\]
Consequently,
\[
 \sum_z\mu(z)a(x-z)
 \ge\sum_{k\ge1}\mu(z_k)\left(\frac1\pi\log k-C\right).
\]
The coefficients on the right increase with $k$. Subject to
$0\le\mu(z_k)\le1$ and $\sum_k\mu(z_k)=M$, by the rearrangement inequality, this sum is
minimized by filling the first $\lfloor M\rfloor$ positions and
placing the remaining mass at the next position. Summing $\log k$ proves
\eqref{eq:packing} for $M\ge1$. For $M<1$, it follows from $a\ge0$.
Applying this estimate to $\mu/K$ and multiplying by $K$ proves
\eqref{eq:packing-scaled}.
\end{proof}

\section{Uniform excitation and range estimates}\label{sec:range}

Throughout this section, $\EE=\EE_0^\eta$ denotes expectation for
the walk started at the origin in the environment $\eta$, and all
bounds are uniform over $\eta\in\cE_{A,B}$. Since $\cE_{A,B}$ is
translation-invariant, this also covers every starting point.

We first treat nonnegative cookies, for which the quadratic form
in Lemma \ref{lem:anisotropic} appears directly. We then separate
the two signs to prove the excitation part of
Proposition \ref{prop:occupation}. A potential generated by all
visited vertices gives its range part.

\subsection{Nonnegative cookies}\label{sec:onesign}

Suppose first that every cookie is nonnegative and that the total
strength at each vertex is at most one. Recall the consumed
measure $\mu_n$ and its potential $\Phi_n$ from \eqref{defPhin}.
The total mass is $\Xi_n$, and the update
$\mu_{j+1}=\mu_j+b_j\delta_{S_j}$ adds the current cookie just
before the next step. Since $a(S_{j+1}-S_j)=1$, the newly added
mass contributes $b_j$ to $\Phi_{j+1}$. Applying
\eqref{EcondincaSz} to the existing measure $\mu_j$ gives
\begin{equation}\label{eq:potential-increment}
 \EE[\Phi_{j+1}-\Phi_j\mid\FF_j]
 =\mu_j(S_j)+b_j+b_j(d*\mu_j)(S_j),
\end{equation}
where $(d*\mu)(x)=\sum_zd(x-z)\mu(z)$. The budget gives
$\mu_j(S_j)+b_j\le1$. Evenness of $d$ and $d(0)=0$ yield
the pathwise identity
\begin{equation}\label{eq:quadratic-identity}
 \sum_{j<n}b_j(d*\mu_j)(S_j)
 =\sum_{i<j<n}b_ib_jd(S_j-S_i)
 =\frac12\sum_{z,w}\mu_n(z)\mu_n(w)d(z-w).
\end{equation}
By \eqref{eq:quadratic-bound} and $0\le\mu_n\le1$, its absolute
value is at most $\Xi_n/2$. Summing
\eqref{eq:potential-increment} and using $\Phi_0=0$ and
$\Xi_n\le n$ therefore gives $\EE\Phi_n\le3n/2$.

\subsection{Cookies of both signs}

To estimate $\Xi_n$ defined by \eqref{defXingen}, we record the two signs separately. Terms involving cookies of the same sign are still controlled by quadratic forms, while
\eqref{eq:signed-rows} controls terms involving opposite signs.

\begin{proof}[Proof of the excitation bound in Proposition \ref{prop:occupation}]
Fix $A,B\ge0$ with $AB<1$.  We shall prove
\begin{equation}\label{eq:excitation-bound}
 \EE\Xi_n\le C\frac{A+B+2AB}{1-AB}
                    \frac n{\log(n+2)}.
\end{equation}
Write $b_j^+=(b_j)_+$ and
$b_j^-=(b_j)_-$, and set
\[
 \mu_n^\pm(z)=\sum_{j<n}b_j^\pm\one_{\{S_j=z\}},\qquad
 \Phi_n^\pm=\sum_z\mu_n^\pm(z)a(S_n-z).
\]
Their total masses are $\Xi_n^\pm=\sum_{j<n}b_j^\pm$, so
$\Xi_n^++\Xi_n^-=\Xi_n\le n$. Also,
$0\le\mu_n^+\le A$ and $0\le\mu_n^-\le B$.
The same calculation as in \eqref{eq:potential-increment} gives
\begin{equation}\label{eq:signed-increment}
 \EE[\Phi_{j+1}^\pm-\Phi_j^\pm\mid\FF_j]
 =\mu_j^\pm(S_j)+b_j^\pm+b_j(d*\mu_j^\pm)(S_j).
\end{equation}
The first two terms are at most $A$ for the plus sign and at
most $B$ for the minus sign. Put $m_\pm=\EE\Xi_n^\pm$ and
$H_n=\pi^{-1}\log(n+2)+C$, with $C$ from
\eqref{eq:signed-rows}. For the plus potential, the same-sign
terms satisfy
\[
 \sum_{j<n}b_j^+(d*\mu_j^+)(S_j)
 =\frac12\sum_{z,w}\mu_n^+(z)\mu_n^+(w)d(z-w)
 \le\frac A2\Xi_n^+.
\]
Here we used \eqref{eq:quadratic-bound} and
$\sum_z\mu_n^+(z)^2\le A\Xi_n^+$. The opposite-sign terms obey
\[
 -b_j^-(d*\mu_j^+)(S_j)
 \le b_j^-\sum_z\mu_j^+(z)d_-(S_j-z)
 \le AH_nb_j^-,
\]
since the support contains at most $n$ vertices. For the minus
potential, the same-sign contribution is
\[
 -\sum_{j<n}b_j^-(d*\mu_j^-)(S_j)
 =-\frac12\sum_{z,w}\mu_n^-(z)\mu_n^-(w)d(z-w)
 \le\frac B2\Xi_n^-,
\]
by \eqref{eq:quadratic-bound}. The opposite-sign term is
bounded by $BH_nb_j^+$.
Summing \eqref{eq:signed-increment} therefore gives
\begin{equation}\label{eq:signed-upper}
 \begin{aligned}
 \EE\Phi_n^+&\le An+\frac A2m_++AH_nm_-,\\
 \EE\Phi_n^-&\le Bn+\frac B2m_-+BH_nm_+.
 \end{aligned}
\end{equation}

For $A>0$, the packing estimate \eqref{eq:packing-scaled} and
the elementary inequality
\[
 x\log(x/A)\ge x\log(n+2)-\frac{A(n+2)}e,
 \qquad x\ge0,
\]
give
\[
 \EE\Phi_n^+\ge\frac{m_+}{\pi}\log(n+2)-CAn.
\]
We used $m_+\le An$ to absorb the linear term in the packing estimate. The same argument applies to the minus sign. When $A=0$ or $B=0$ the corresponding mass vanishes, so its lower bound holds
directly. Comparing with \eqref{eq:signed-upper} and using
$m_++m_-\le n$ yields
\begin{equation}\label{eq:two-sign-system}
 \begin{aligned}
 (m_+-Am_-)\log(n+2)&\le CAn,\\
 (m_--Bm_+)\log(n+2)&\le CBn.
 \end{aligned}
\end{equation}
Multiply the first inequality by $1+B$ and the second by $1+A$,
and add. Since $AB<1$, we obtain \eqref{eq:excitation-bound}.

For the recurrence proof, we also retain the leading coefficient.
The bound just proved gives $m_\pm=O_{A,B}(n/\log n)$.
For $A>0$, Jensen's inequality improves the packing lower bound \eqref{eq:packing-scaled} to
\[
 \EE\Phi_n^+
 \ge\frac{m_+}{\pi}\log\frac{m_+}{A}-Cm_+
 =\frac{m_+}{\pi}\log n-o_{A,B}(n).
\]
Indeed, $x\log(An/x)$ is increasing for $0<x\le An/e$, so the
coarse bound gives
$m_+\log(An/m_+)=O_{A,B}(n\log\log n/\log n)=o(n)$. Comparing
this lower bound and its minus counterpart with
\eqref{eq:signed-upper} gives
\[
 \begin{aligned}
 (m_+-Am_-)\log n&\le\pi An+o_{A,B}(n),\\
 (m_--Bm_+)\log n&\le\pi Bn+o_{A,B}(n).
 \end{aligned}
\]
Again, when $A=0$ (respectively $B=0$), the first (respectively second) inequality holds directly. Set
$\Lambda(A,B)=\pi(A+B+2AB)/(1-AB)$. Multiplying the two
inequalities by $1+B$ and $1+A$, respectively, and adding proves
\begin{equation}\label{eq:leading-excitation}
 \sup_{\eta\in\cE_{A,B}}\sup_{x\in\ZZ^2}
 \EE_x^\eta\Xi_n
 \le\bigl(\Lambda(A,B)+o_{A,B}(1)\bigr)\frac n{\log n}
 \qquad(n\longrightarrow\infty).
\end{equation}
For fixed $A,B$, every error above is uniform over the environment
and starting point. These estimates therefore also apply to the
remaining environment after any finite history, with the same
error bound.
\end{proof}

\subsection{The full range}

Some vertices may have only zero-strength cookies, so the
consumed excitation need not count every visited vertex.
We therefore place one unit of mass at each visited vertex.
The excitation estimate just proved controls the correction
to this potential.

\begin{proof}[Proof of the range bound and uniformity in Proposition \ref{prop:occupation}]
Write $\mathcal V_j=\{S_0,\ldots,S_j\}$ for the visited set
through time $j$, and thus $|\mathcal V_j|=R_{j+1}$. Define
\[
 V_j=\sum_{z\in\mathcal V_j}a(S_j-z).
\]
Note that $S_{j+1}$ contributes $a(S_{j+1}-S_{j+1})=a(0)=0$ to $V_{j+1}$ even when it is newly visited. By \eqref{EcondincaSz}, we obtain the exact identity
\[
 \EE[V_{j+1}-V_j\mid\FF_j]
 =\sum_{z\in\mathcal V_j}\EE [a(S_{j+1}-z)-a(S_j-z) \mid \FF_j] =1+b_j(d*\one_{\mathcal V_{j}})(S_j),
\]
Using the last assertion in Lemma \ref{lem:anisotropic} and $|\mathcal V_{j}|\le n$ for $j\leq n-1$, we obtain $|(d*\one_{\mathcal V_{j}})(S_j)| \leq C \log(n+2)$ and thus
\[
 \EE V_n\le n+1+C\log(n+2)\EE\Xi_n
 \le C\frac{(1+A)(1+B)}{1-AB}\,n,
\]
where we also used \eqref{eq:excitation-bound} in the last inequality. By the packing estimate, applied to $\one_{\mathcal V_n}$,
\[
 \EE\bigl[|\mathcal V_n|\log(1+|\mathcal V_n|)\bigr]
 \le C\frac{(1+A)(1+B)}{1-AB}\,n.
\]
For $n\ge3$, split according to $|\mathcal V_n|\le\sqrt n$ to get
\[
 \EE R_n\le\sqrt n+
 \frac{2}{\log n}\EE\bigl[|\mathcal V_n|\log(1+|\mathcal V_n|)\bigr].
\]
This gives the range bound with the stated choice of $C_{A,B}$,
after increasing the universal constant for small $n$. Combining
it with \eqref{eq:excitation-bound} proves
\eqref{eq:main-occupation}.

Every estimate above is independent of the starting point. To obtain
the conditional assertion, fix an
almost surely finite stopping time $T$. Delete from each stack the
cookies used before time $T$ and translate the current position $S_T$
to the origin. Deleting cookies only weakens the bounds
\eqref{eq:budgets}, and translation leaves them unchanged, so the
remaining environment still belongs to $\cE_{A,B}$. Conditionally
on $\FF_T$, the walk after time $T$ is governed by the same law
\eqref{eq:cookie-law} in this admissible environment, and the
finite-time calculations above apply with the same constants. In
this assertion, the range counts the distinct
vertices of the future segment $S_T,\ldots,S_{T+n-1}$, whether or
not they were visited before $T$.
\end{proof}

We shall also use the following consequence of the conditional
occupation estimate.

\begin{lemma}\label{lem:occupation-moments}
Suppose that $A,B\ge0$ and $AB<1$. There is a constant
$C_{A,B}<\infty$ such that, uniformly over
$\eta\in\cE_{A,B}$ and starting points $x$, for all integers
$n,k\ge1$,
\[
 \EE_x^\eta\bigl[\Xi_n^k+R_n^k\bigr]
 \le 2k!\left(\frac{C_{A,B}n}{\log(n+2)}\right)^k.
\]
The same bound holds conditionally on $\FF_T$ after any almost
surely finite stopping time $T$, with $\Xi_n$ and $R_n$ replaced
by the excitation and the range of the segment
$S_T,\ldots,S_{T+n-1}$.
\end{lemma}

\begin{proof}
Fix $n$, and put $u=C_{A,B}n/\log(n+2)$, with the constant
from Proposition \ref{prop:occupation}. Consider either
$a_j=|b_j|$ or
\[
 a_j=\one_{\{S_j\notin\{S_0,\ldots,S_{j-1}\}\}},
 \qquad 0\le j<n,
\]
where the set is empty when $j=0$. In both cases, $a_j$ is
$\FF_j$-measurable and belongs to $[0,1]$. Write
$A_i=\sum_{j=i}^{n-1}a_j$. For the second choice, $A_i$
counts vertices first visited during this interval, so it is
at most $|\{S_i,\ldots,S_{n-1}\}|$. Thus the conditional
estimate in Proposition \ref{prop:occupation} gives, for both choices,
\[
 \EE[A_i\mid\FF_i]\le u,\qquad 0\le i<n.
\]
Here we used that $t/\log(t+2)$ is increasing for $t\ge0$.

We prove by induction on $k$ that
$\EE[A_i^k\mid\FF_i]\le k!u^k$ for every $i<n$.
The case $k=1$ was just established. Since
$A_j-A_{j+1}=a_j$ and $A_n=0$, telescoping and the mean value theorem yield
\[
 A_i^k= \sum_{j=i}^{n-1} (A_j^k-A_{j+1}^k) \le k\sum_{j=i}^{n-1}a_jA_j^{k-1}.
\]
For $k\ge2$, the induction hypothesis and the
$\FF_j$-measurability of $a_j$ therefore imply
\[
 \begin{aligned}
 \EE[A_i^k\mid\FF_i]
 &\le k\sum_{j=i}^{n-1}
   \EE\bigl[a_j\EE[A_j^{k-1}\mid\FF_j]\mid\FF_i\bigr]\\
 &\le k!u^{k-1}\EE[A_i\mid\FF_i]
 \le k!u^k.
 \end{aligned}
\]
Taking $i=0$ for the two choices of $a_j$ proves the assertion.
After a stopping time $T$, the same argument uses the filtration
$(\FF_{T+j})_{j\ge0}$ and counts first visits relative to the
segment starting at $T$. At each later time these indicators are
still measurable from the past, and their remaining sum is bounded
by the range of the remaining segment. The conditional assertion
therefore follows from the same estimate.
\end{proof}

\section{The range of M(2,1,2)}\label{sec:asymptotic}

For the initially all-fresh walk, each visited vertex contributes exactly
one unit to the consumed measure. The growing potential from
Section \ref{sec:onesign} therefore records the entire range. We show
that this potential is asymptotic to $n$, and that it differs from
$R_n\log n/\pi$ by $o(n)$. These two comparisons determine the
leading constant in the range. 

\begin{proposition}\label{prop:range-limit}
For the walk $M(2,1,2)$, the convergence \eqref{eq:m212-range} holds. Moreover, for all sufficiently large $n$,
\begin{equation}\label{eq:range-rates}
 \EE\left|\frac{\log n}{n}R_n-\pi\right|
 \le\frac C{\sqrt{\log n}},\qquad
 \left|\frac{\log n}{n}\EE R_n-\pi\right|
 \le C\frac{\log\log n}{\log n}.
\end{equation}
\end{proposition}

\begin{proof}
In the notation of Section \ref{sec:onesign},
\[
 \mu_n=\mathbf1_{\{S_0,\ldots,S_{n-1}\}},\qquad
 \Phi_n=\sum_z\mu_n(z)a(S_n-z).
\]
At time $j$, if the site $S_j$ is newly visited, i.e., $S_j \notin \{S_0,\ldots,S_{j-1}\}$, then $\mu_j(S_j)=0$ and $b_j=1$; otherwise, $\mu_j(S_j)=1$ and $b_j=0$. Thus the first two terms on the right-hand side of \eqref{eq:potential-increment} always sum to one.
Put
\[
 Q_n=\sum_{z,w}\mu_n(z)\mu_n(w)d(z-w).
\]
Then \eqref{eq:potential-increment} and \eqref{eq:quadratic-identity} give
\[
 Q_{j+1}-Q_j=2b_j(d*\mu_j)(S_j),\qquad
 \EE[\Phi_{j+1}-\Phi_j\mid\FF_j]
 =1+\frac12(Q_{j+1}-Q_j).
\]
In particular, $Q_{j+1}-Q_j$ is $\FF_j$-measurable, and
\begin{equation}\label{eq:asymptotic-martingale}
 \mathcal M_n=\Phi_n-n-\frac12Q_n
\end{equation}
is a martingale starting at zero. All quantities have finite moments
at each fixed time, since the walk and the sources lie in a deterministic
finite ball.

We first prove the two comparisons in $L^1$:
\[
 \EE|\Phi_n-n|=o(n),\qquad
 \EE\left|\Phi_n-\frac{R_n}{\pi}\log n\right|=o(n).
\]
The potential-kernel asymptotic \eqref{eq:potential-asymptotic} gives
\begin{equation}\label{eq:potential-first-difference}
 |a(z+e)-a(z)|\le\frac{C}{1+|z|}
\end{equation}
for every nearest-neighbour vector $e$. Ordering any set of $R$ lattice points by their distance from a fixed
point shows that the sum of the right-hand side
of \eqref{eq:potential-first-difference} over that set is at most
$C\sqrt R$. Here the $k$-th point contributes at most $C/\sqrt k$,
because a disk of radius $r$ contains at most $C(1+r)^2$ lattice points. By definition, $\Phi_{j+1}-\Phi_j=b_j+X_{j+1}$ where
\[
 X_{j+1}:=\sum_{z\in \{S_0, \dots, S_{j-1} \}} \bigl(a(S_{j+1}-z)-a(S_j-z)\bigr).
\]
Thus $|X_{j+1}|\le C\sqrt{R_j}$. Moreover, the martingale increment is
\[
\mathcal M_{j+1}-\mathcal M_j=\Phi_{j+1}-\Phi_j-\EE[\Phi_{j+1}-\Phi_j\mid\FF_j]=X_{j+1}-\EE[X_{j+1}\mid\FF_j]
\]
 In particular,
\begin{equation}\label{eq:range-martingale-increments}
 \begin{aligned}
 |\mathcal M_{j+1}-\mathcal M_j|&\le C\sqrt{R_j},\\
 \EE\bigl[(\mathcal M_{j+1}-\mathcal M_j)^2\mid\FF_j\bigr]
 &=\operatorname{Var}(X_{j+1}\mid\FF_j)\le CR_j.
 \end{aligned}
\end{equation}
The first increment is zero, with the convention $R_0=0$.
Orthogonality of martingale increments and
Proposition \ref{prop:occupation} now give
\[
 \EE\mathcal M_n^2
 \le C\sum_{j<n}\EE R_j
 \le \frac{C n^2}{\log(n+2)}.
\]
Moreover, \eqref{eq:quadratic-bound} implies $|Q_n|\le R_n$.
Using Proposition \ref{prop:occupation} once more in
\eqref{eq:asymptotic-martingale}, we obtain
\begin{equation}\label{eq:potential-normalization}
 \EE|\Phi_n-n|\le \EE |\mathcal M_n| +\frac{1}{2} \EE |Q_n| \le \frac{Cn}{\sqrt{\log(n+2)}}.
\end{equation}

To compare $\Phi_n$ with $R_n\log n/\pi$, the packing estimate
supplies a lower bound on the potential, while a maximal displacement
estimate supplies an upper bound. By \eqref{eq:packing},
\begin{equation}\label{eq:potential-comparison-pathwise}
 \left(\frac{R_n}{\pi}\log n-\Phi_n\right)_+
 \le\frac{R_n}{\pi}\log\frac{n}{R_n}+CR_n.
\end{equation}
Write $r_n=\EE R_n$. Since $x\mapsto x\log(n/x)$ is concave,
Jensen's inequality bounds the expectation of the right-hand side from above by $r_n\log(n/r_n)/\pi+Cr_n$.
Proposition \ref{prop:occupation} gives $r_n\le Cn/\log (n+2)$, which is less than $n/e$ for all large $n$. Observe that the function $x\mapsto x\log(n/x)$ is increasing on $[0,n/e]$. It follows that
\begin{equation}\label{eq:potential-comparison-lower}
 \EE\left[\left(\frac{R_n}{\pi}\log n-\Phi_n\right)_+\right]
 \le C\frac{n\log\log n}{\log n}.
\end{equation}

For the other side, both coordinates of the walk $S$ are martingales with
increments bounded by one. The maximal Azuma inequality (see e.g. \cite[Theorem 3.2.1]{MR5088005}) therefore yields
\begin{equation}\label{eq:range-maximal-displacement}
 \PP\left(\max_{j\le n}|S_j|>\sqrt n\log n\right)
 \le C\exp\{-c(\log n)^2\}.
\end{equation}
On the complementary event, every source in $\Phi_n$ lies within
distance $2\sqrt n\log n$ of $S_n$. By
\eqref{eq:potential-asymptotic},
\begin{equation}
    \label{PhinupRnAzuma}
    \Phi_n\le R_n\left(\frac1\pi\log n
                  +\frac2\pi\log\log n+C\right). 
\end{equation}
On the exceptional event, the deterministic bound
$\Phi_n\le Cn\log(n+2)$ suffices since its contribution to the expectation
is $o(1)$. Thus
\[
 \EE\left[\left(\Phi_n-\frac{R_n}{\pi}\log n\right)_+\right]
 \le C\frac{n\log\log n}{\log n}+o(1).
\]
Together with \eqref{eq:potential-comparison-lower}, this proves
\begin{equation}\label{eq:potential-comparison-L1}
 \EE\left|\Phi_n-\frac{R_n}{\pi}\log n\right|
 \le C\frac{n\log\log n}{\log n}
\end{equation}
for all sufficiently large $n$. Combining this bound with
\eqref{eq:potential-normalization} proves convergence in $L^1$
and the first estimate in \eqref{eq:range-rates}.
For the second estimate, observe that
\[
\left|\frac{\log n}{n}\EE R_n-\pi\right| \le \frac{\pi}{n}\EE\left|\frac{R_n}{\pi}\log n-\Phi_n\right| + \frac{\pi}{n} |\EE \Phi_n -n| \le C\frac{\log\log n}{\log n},
\]
where we took expectations in \eqref{eq:asymptotic-martingale} to obtain
\[
 |\EE\Phi_n-n|=\frac12|\EE Q_n|
 \le\frac12\EE R_n\le\frac{Cn}{\log(n+2)}.
\]
Lemma \ref{lem:occupation-moments} gives, for every integer $k\ge1$,
\[
 \sup_{n\ge2}\EE\left(\frac{\log n}{n}R_n\right)^k<\infty.
\]
Choosing $k>p$ makes the $p$-th powers uniformly integrable.
The convergence in probability already proved therefore yields
convergence in $L^p$ for every $1\le p<\infty$.

It remains to prove almost sure convergence. By
\eqref{eq:range-martingale-increments}, the sum of squared
martingale increments satisfies
\[
 [\mathcal M]_n
 :=\sum_{j<n}(\mathcal M_{j+1}-\mathcal M_j)^2
 \le C\sum_{j<n}R_j\le CnR_n.
\]
The maximal Burkholder--Davis--Gundy inequality (see \cite[Theorem 1.1]{MR400380}), followed by
Lemma \ref{lem:occupation-moments} with $k=2$, gives
\[
 \EE\max_{j\le n}|\mathcal M_j|^4
 \le C\EE[\mathcal M]_n^2
 \le Cn^2\EE R_n^2
 \le\frac{Cn^4}{\log^2(n+2)}.
\]
For each $\varepsilon>0$, the resulting bounds on
\[
 \PP\left(\max_{j\le2^k}|\mathcal M_j|>
                  \varepsilon2^k\right)\le\frac{C}{\varepsilon^4 k^2}
\]
are summable in $k$. Borel--Cantelli and
$2^{k-1}<n\le2^k$ show that $\mathcal M_n/n\to0$ almost surely.

Since $|Q_n|\le R_n\le n$, the decomposition
\eqref{eq:asymptotic-martingale} first gives $\Phi_n=O(n)$
almost surely. The packing bound then implies
$R_n=O(n/\log n)$ almost surely: for $R_n>\sqrt n$, it gives
\[
 \Phi_n\ge R_n\left(\frac1{2\pi}\log n-C\right),
\]
and $R_n\le\sqrt n$ already has the required bound.
Consequently $Q_n/n\to0$, and
$\Phi_n/n\to1$ almost surely.
The probabilities in \eqref{eq:range-maximal-displacement} are
summable over all $n$. Borel--Cantelli therefore gives
\[
 \max_{j\le n}|S_j|\le\sqrt n\log n
 \quad\text{for all sufficiently large }n,
 \qquad\text{a.s.}.
\]
On this event, the upper potential bound \eqref{PhinupRnAzuma} together with
$R_n=O(n/\log n)$, show that
\[
 \left(\Phi_n-\frac{R_n}{\pi}\log n\right)_+=o(n)
 \qquad\text{a.s.}.
\]
For the lower comparison, use \eqref{eq:potential-comparison-pathwise}, the monotonicity of $x\log(n/x)$ on $[0,n/e]$, and $R_n=O(n/\log n)$. This proves the almost sure range limit
and completes the proof.
\end{proof}

\begin{remark}\label{rem:constant-first-cookie}
The same range limit and the estimates in \eqref{eq:range-rates}
hold when the first cookie at every vertex has a common strength $\theta\in[-1,1]$ and all later cookies are zero. To see this, let $I_j$ indicate that $S_j$ is fresh and use unit mass on the
visited set,
\[
 \nu_n=\mathbf1_{\{S_0,\ldots,S_{n-1}\}},\qquad
 \Phi_n^\nu=\sum_z\nu_n(z)a(S_n-z),\qquad
 Q_n^\nu=\sum_{z,w}\nu_n(z)\nu_n(w)d(z-w).
\]
The measure $\nu_n$ assigns unit mass to each visited vertex,
independently of $\theta$. Since $b_j=\theta I_j$ and
$\nu_{j+1}=\nu_j+I_j\delta_{S_j}$,
\[
 \EE[\Phi_{j+1}^\nu-\Phi_j^\nu\mid\FF_j]
 =1+\theta I_j(d*\nu_j)(S_j),\qquad
 Q_{j+1}^\nu-Q_j^\nu=2I_j(d*\nu_j)(S_j).
\]
Thus $\Phi_n^\nu-n-(\theta/2)Q_n^\nu$ is a martingale. The remaining
estimates are unchanged, uniformly over $\theta\in[-1,1]$, since these environments satisfy $A=\theta_+$, $B=\theta_-$ and
$A+B\le1$.
\end{remark}

\section{Recurrence}\label{sec:recurrence}

Write $\kappa=A+B$ and let $\Lambda=\Lambda(A,B)$ be the
coefficient in \eqref{eq:leading-excitation}, with $AB<1$.
We compare the negative ordinary drift of $\log\log|z|$ with
the excitation drift, first on local blocks and then up to exit
from a large ball. Related potential and martingale arguments are used
by Chan \cite{MR4575009} for horizontal-vertical walks.

Fix $r_0=e^3$. Let $\psi:\RR^2\to[1,\infty)$ be a smooth radial function with $\psi(z)=\log \log |z|$ for $|z|\geq r_0$.
 Put $w(z)=|z|^{-2}(\log|z|)^{-2}$ for $|z|\ge r_0$, and $w(z)=0$
otherwise. 

\begin{lemma}\label{lem:loglog}
As $|z|=r\to\infty$, uniformly in the argument $\vartheta$ of $z$,
\begin{equation}\label{eq:loglog-drift}
 \begin{aligned}
 (P-I)\psi(z)&=-\tfrac14w(z)+o(w(z)),\\
 D\psi(z)&=-\frac{\cos(2\vartheta)}{2r^2\log r}+O(w(z)).
 \end{aligned}
\end{equation}
Moreover,
\begin{equation}\label{eq:loglog-budget}
 \sum_{|z|<R}(D\psi(z))_\pm=\log\log R+O(1).
\end{equation}
\end{lemma}

\begin{proof}
For $r$ sufficiently large, the radial Laplacian gives
\[
 \Delta\psi(z)
 =\frac{d^2}{dr^2}\log\log r
   +\frac1r\frac{d}{dr}\log\log r
 =-\frac1{r^2\log^2r}.
\]
For each coordinate direction $i=1,2$, Taylor's formula gives
\[
 \psi(z+e_i)-2\psi(z)+\psi(z-e_i)
 =\partial_i^2\psi(z)+O(r^{-4}/\log r).
\]
Taking one quarter of the sum and difference of these two
expressions gives, respectively,
\[
 \begin{aligned}
 (P-I)\psi(z)
 &=\tfrac14\Delta\psi(z)+O(r^{-4}/\log r),\\
 D\psi(z)
 &=\tfrac14(\partial_1^2-\partial_2^2)\psi(z)+O(r^{-4}/\log r).
 \end{aligned}
\]
For $z=(r\cos\vartheta,r\sin\vartheta)$, the second identity becomes
\[
 D\psi(z)
 =-\frac{\cos(2\vartheta)}{2r^2\log r}
  -\frac{\cos(2\vartheta)}{4r^2\log^2r}
  +O(r^{-4}/\log r).
\]
This proves \eqref{eq:loglog-drift}. The difference between
$D\psi(z)$ and $-\cos(2\vartheta)/(2r^2\log r)$ is $O(w(z))$
and is summable over the lattice, since $\sum_z w(z)<\infty$.
The angular integral of each of $(\cos(2\vartheta))_+$ and
$(\cos(2\vartheta))_-$ equals $2$, so the integral of either
sign of the leading term is
$\int_{r_0}^Rdr/(r\log r)=\log\log R+O(1)$.
Replacing the integral by a lattice sum costs a bounded amount:
the variation on a unit square at distance $r$ is at most
$C/(r^3\log r)$, and these errors are summable. This proves
\eqref{eq:loglog-budget}.
\end{proof}

Write $W_n=\sum_{k<n}w(S_k)$ and we refer to $W$ as the clock.

\begin{lemma}\label{lem:local-drift}
Fix $0<c<1/4$ and $J>\Lambda/4$. There exist
$\epsilon,\delta\in(0,1/2)$ and $r_1,C<\infty$, depending only
on $A,B,c,J$ and $\psi$, such that the following holds uniformly
over all remaining environments. From $x$ with $r=|x|\ge r_1$, put
\[
 m=\lfloor\epsilon\delta^2r^2\rfloor,\qquad
 \zeta=\inf\{n\ge0:|S_n-x|\ge\delta r\},\qquad
 \sigma=m\wedge\zeta.
\]
We call the path up to $\sigma$ a block (from $x$). Then $m\ge1$, and for every stopping time $\nu\le\sigma$,
\begin{equation}\label{eq:local-stopped-drift}
 \EE_x[\psi(S_\nu)-\psi(x)]
 \le-c\EE_xW_\nu+J\EE_xW_\sigma,
 \qquad W_\sigma\le\frac C{\log^2r}.
\end{equation}
The same estimates hold conditionally at the start of each block: for any stopping time $T$ with $|S_T|\geq r_1$, \eqref{eq:local-stopped-drift} holds with $x=S_T$ and the environment conditioned on $\FF_T$.
\end{lemma}

\begin{proof}
Choose $\epsilon,\delta$ sufficiently small that
\[
 \frac{\Lambda}{4(1-\epsilon)}
       \left(\frac{1+\delta}{1-\delta}\right)^2<J.
\]
Applying the optional stopping theorem to the martingale $(|S_n-x|^2-n)$ at $\sigma$ and Markov inequality give
$\PP_x(\zeta\le m)\le m/(\delta r)^2\le\epsilon$, and hence
$\EE_x\sigma\ge(1-\epsilon)m$.
Before $\sigma$, the norm of the walk lies between $(1-\delta)r$ and
$(1+\delta)r$. Consequently,
\[
 \EE_xW_\sigma
 \ge\left((1+\delta)^{-2}+o(1)\right)
       \frac{(1-\epsilon)m}{r^2\log^2r}.
\]
Since $\sigma\le m$, the excitation term satisfies
\[
 \EE_x\sum_{k<\sigma}|b_kD\psi(S_k)|
 \le\sup_{|z-x|<\delta r}|D\psi(z)|\,\EE_x\Xi_m.
\]
We can therefore use the deterministic-time estimate
\eqref{eq:leading-excitation}, which gives
$\EE_x\Xi_m\le(\Lambda+o(1))m/\log m$ for large $m$ uniformly over the
environment. Since $\log m\sim2\log r$, Lemma
\ref{lem:loglog} and the bound on $\EE_xW_\sigma$ yield
\[
 \EE_x\sum_{k<\sigma}|b_kD\psi(S_k)| \le\frac{\Lambda+o(1)}{4(1-\delta)^2}
               \frac{m}{r^2\log^2r} \le\left\{
     \frac{\Lambda}{4(1-\epsilon)}
       \left(\frac{1+\delta}{1-\delta}\right)^2+o(1)
     \right\}\EE_xW_\sigma.
\]
The errors are uniform over centers and remaining environments.
Choose $r_1\geq r_0/(1-\delta)$ large enough such that for all $r\geq r_1$, one has $m\ge1$ and the last coefficient at most $J$, and $(P-I)\psi\le-cw$ throughout
the ball $B(x,\delta r)$ for any $x$ with $|x|=r$. Applying the optional stopping theorem to the martingale
\begin{equation}
    \label{MartpsiPbD}
    M_n=\psi(S_n)-\psi(x) - \sum_{k<n}(P-I+b_kD)\psi(S_k),
\end{equation}
and using $\nu\le\sigma\le m$, we have
\[
 \begin{aligned}
 \EE_x[\psi(S_\nu)-\psi(x)]
 &=\EE_x\sum_{k<\nu}\bigl((P-I)\psi(S_k)+b_kD\psi(S_k)\bigr)\\
 &\le-c\EE_xW_\nu
       +\EE_x\sum_{k<\sigma}|b_kD\psi(S_k)|.
 \end{aligned}
\]
The preceding estimate proves the first inequality in
\eqref{eq:local-stopped-drift}.
The second follows from $\sigma\le m$ and
$\sup_{|z-x|<\delta r}w(z)\le C/(r^2\log^2r)$.
At a block starting at a stopping time $T$, the time limit $m$
is $\FF_T$-measurable. Thus the conditional assertion follows from the uniformity
of \eqref{eq:leading-excitation} since remaining stacks still satisfy
the same budgets.
\end{proof}

\begin{proof}[Proof of Theorem \ref{thm:main}]
Assume $\kappa<1+1/(2\pi+1)$, so that
$AB\le\kappa^2/4<1$. We first choose $c,J$ as in Lemma
\ref{lem:local-drift}, with $J>c$, such that
\begin{equation}\label{eq:recurrence-criterion}
 q=\rho\kappa<1,\qquad \rho=1-\frac cJ.
\end{equation}
For $\kappa<1$, any admissible $c,J$ with $J>c$ work. For $\kappa\ge1$,
\begin{equation}
    \label{eqcondreckaLa}
     \Lambda\le\frac{2\pi\kappa}{2-\kappa},\qquad
 \kappa\left(1-\frac1\Lambda\right)
 \le\kappa-\frac{2-\kappa}{2\pi}<1.
\end{equation}
Here $\Lambda\ge\pi\kappa>1$, so taking $c$ sufficiently close
to $1/4$ from below and $J$ sufficiently close to $\Lambda/4$
from above gives \eqref{eq:recurrence-criterion}.

Fix the constants of Lemma \ref{lem:local-drift}, and increase
$r_1$ so that $\psi(z)=\log\log|z|$ and $(P-I)\psi(z) \leq -c w(z)$ for $|z|\ge r_1$.
For $r_1\le|x|<R$, let
\[
 T_{r_1}=\inf\{n\ge0:S_n\in B_{r_1}\},\qquad
 \tau_R=\inf\{n\ge0:|S_n|\ge R\},\qquad
 \tau=T_{r_1}\wedge\tau_R.
\]
Stopping $|S_n|^2-n$ first at $n\wedge\tau_R$ gives
$\EE_x\tau\le\EE_x\tau_R\le(R+1)^2$. All stopped drift
sums below are therefore integrable. Put
$D_\tau=\EE_x\psi(S_\tau)-\psi(x)$.

The signed excitation consumed at a vertex before $\tau$ lies
in $[-B,A]$, and all sources lie in $B_R$. Consequently,
\[
 \sum_{k<\tau}b_kD\psi(S_k)
 \le A\sum_{|z|<R}(D\psi(z))_+
       +B\sum_{|z|<R}(D\psi(z))_-.
\]
Again, we apply the optional stopping theorem to the martingale in \eqref{MartpsiPbD} and use \eqref{eq:loglog-budget} to obtain
\begin{equation}\label{eq:global-drift}
 D_\tau\le\kappa\log\log R-c\EE_xW_\tau+C,
\end{equation}
where we also used that $(P-I)\psi(z) \leq -c w(z)$ for $|z|\ge r_1$. We note that all constants in this proof are independent of $x,R$ and the remaining environment.

For a second bound, start at $T_0=0$ and concatenate the blocks of Lemma \ref{lem:local-drift} until time $\tau$. More precisely, if $T_j<\tau$ (in particular, $|S_{T_j}|\ge r_1$), define
the next block from $S_{T_j}$ as follows. Define
$$
 m_j=\lfloor\epsilon\delta^2|S_{T_j}|^2\rfloor,
$$
let $\zeta_j$ be the first time $n\ge T_j$ with
$|S_n-S_{T_j}|\ge\delta|S_{T_j}|$, and put
$$
 T_{j+1}=T_j+\bigl(m_j\wedge(\zeta_j-T_j)\bigr).
$$
Let $N$ be the first index with $T_N\ge\tau$, and freeze
$T_j=T_N$ for $j\ge N$. Each active block has length at least
one, so $N\le\tau<\infty$. Its activity event
$\{j<N\}=\{T_j<\tau\}$ is measurable at $T_j$.
On this event, apply \eqref{eq:local-stopped-drift} conditionally
at $T_j$, with both clocks counted from $T_j$ and the observed
part of the block stopped at $T_{j+1}\wedge\tau$. This gives
\begin{equation}
    \label{eq:psiWubd}
  \EE_x\bigl[\psi(S_{T_{j+1}\wedge\tau})-\psi(S_{T_j})
                   \mid\FF_{T_j}\bigr]\le-c\EE_x\bigl[W_{T_{j+1}\wedge\tau}-W_{T_j}
                   \mid\FF_{T_j}\bigr]
       +J\EE_x\bigl[W_{T_{j+1}}-W_{T_j}
                   \mid\FF_{T_j}\bigr].   
\end{equation}
The negative drift is counted only up to $\tau$, while the
positive term retains the full block.
Only the last full block can extend beyond $\tau$, and its
entire clock is bounded by $C/\log^2|S_{T_{N-1}}|\le C$.
Thus, pathwise, $0\le W_{T_N}-W_\tau\le C$.
Multiplying the conditional inequalities by $\one_{\{j<N\}}$,
taking expectations, and summing gives
\begin{equation}\label{eq:local-drift-summed}
 D_\tau\le-c\EE_xW_\tau+J\EE_xW_{T_N}
          \le(J-c)\EE_xW_\tau+C,
\end{equation}
where in the first inequality, we first summed over $j \leq K$ in \eqref{eq:psiWubd} and let $K \to \infty$ using the bounded and monotone convergence theorems. Multiply \eqref{eq:local-drift-summed} by $c/J$ and \eqref{eq:global-drift} by $1-c/J$, and add. The clock terms cancel, leaving
\begin{equation}\label{eq:annular-expectation}
 \EE_x\psi(S_\tau)\le\psi(x)+q\log\log R+C.
\end{equation}
Since $\psi\ge0$ and $\psi(S_\tau)\ge\log\log R$ on the event $\{\tau=\tau_R\}$,
\[
 \PP_x(\tau_R<T_{r_1})
 \le q+\frac{\psi(x)+C}{\log\log R}.
\]
Note that each $\tau_R$ is almost surely finite as $\EE_x\tau_R\le(R+1)^2$. Letting $R\to\infty$ gives $\PP_x(T_{r_1}=\infty)\le q<1$.

This bound holds for every starting point outside $B_{r_1}$ and
every remaining environment. For $E=\{T_{r_1}=\infty\}$,
conditioning at time $n$ therefore gives
\[
 \PP_x(E\mid\FF_n)\le q\one_{\{T_{r_1}>n\}}.
\]
By L\'evy's upward theorem, as $n\to \infty$, the left-hand side converges a.s. to $\one_E$, so $\PP_x(E)=0$. Applying this conclusion after every deterministic time shows that $B_{r_1}$ is visited
infinitely often almost surely.

We have proved $\PP(\cup_{z\in B_{r_1}} \{z\text{ is visited i.o.}\})=1$. On the event $\{z\text{ is visited i.o.}\}$, absolute summability of the stacks at $z$ implies that the strengths used on its successive departures from $z$ tend to zero. Each of the four departure probabilities is consequently at least $1/8$ from some visit onward. Conditional Borel--Cantelli lemma gives infinitely many jumps from $z$ to each of its neighbors. Since $\ZZ^2$ is connected, this implies that every vertex is visited i.o. a.s. on $\{z\text{ is visited i.o.}\}$, which completes the proof.
\end{proof}

\begin{remark}\label{rem:one-sign-recurrence}
The condition $AB<1$ is sufficient for the excitation bound in Proposition \ref{prop:occupation}.
As shown in the proof (see especially \eqref{eqcondreckaLa}), the following stronger condition is sufficient for recurrence: 
\begin{equation}
    \label{eqcondrecnkaL}
    (A+B) \left(1-\frac{1}{\Lambda}\right) <1.
\end{equation}
Consequently, since $\Lambda(A,0)=\pi A$, for cookies of one sign the proof gives the larger parameter window $A<1+1/\pi$ when $B=0$, and symmetrically $B<1+1/\pi$ when $A=0$.
\end{remark}

Write $C_z(\ell)=\sum_{j=1}^{\ell}\beta_{z,j}$ for the signed partial sum of the stack at $z$, with $C_z(0)=0$. The following example shows that the  threshold $\kappa_*$ is not optimal.

\begin{example}\label{example:canceling-two-cookies}
 The stack
$(1,-\theta,0,\ldots)$ at every vertex is recurrent for every
$0<\theta<1$. Here $A=1$, $B=\theta$, so $AB<1$ and the local estimate
still holds, while every partial sum $C_z(\ell)$ lies in
$[0,1]$. In the fresh case, grouping the excitation drift by
departure vertex gives
$\sum_{k<\tau}b_kD\psi(S_k)=\sum_z\beta_zD\psi(z)$ with $\tau=T_{r_1}\wedge\tau_R$ as in the proof of Theorem \ref{thm:main}, where $\beta_z\in\{0,1,1-\theta\}\subset[0,1]$ is the signed excitation
consumed at $z$. Since $\beta_z\ge0$, the terms with
$D\psi(z)<0$ are nonpositive and may be dropped, so the sum is
at most $\sum_z(D\psi(z))_+=\log\log R+O(1)$; thus the
coefficient $\kappa$ in \eqref{eq:global-drift} can be replaced
by $1$.
After a finite history, if $\ell_z$ cookies have been used at
$z$ and $N$ further departures from $z$ occur, the net
consumption at $z$ is $C_z(\ell_z+N)-C_z(\ell_z)$. The first
term is again bounded as above, since
$C_z(\ell_z+N)\in[0,1]$, while the subtracted term
$-\sum_z C_z(\ell_z)D\psi(z)$ has absolute value at most the
finite constant $\sum_{z:\ell_z>0}|C_z(\ell_z)||D\psi(z)|$,
which is finite and independent of $R$ because only finitely
many $\ell_z$ are nonzero. Choosing any $0<c<1/4$ and
$J>\max\{c,\Lambda/4\}$ as in the proof above then gives a
limiting avoidance probability at most $1-c/J<1$; the additive constant $C$ in \eqref{eq:global-drift} may depend on the history, but this limiting avoidance
bound does not. The conditional-probability and
neighbor-propagation arguments above complete the proof.
\end{example}

\section{Extensions}\label{sec:extensions}

We give two further examples of recurrence beyond the assumptions of Theorem \ref{thm:main}. Recall that $C_z(\ell)=\sum_{j=1}^{\ell}\beta_{z,j}$, with $C_z(0)=0$.

\begin{proposition}\label{prop:bounded-prefix}
Suppose that there are $u,v\ge0$ with $u+v<1$ such that
\[
 -v\le C_z(\ell)\le u
 \quad\text{for every }z\in\ZZ^2\text{ and }\ell\ge0.
\]
Then the walk almost surely visits every vertex infinitely often,
even if $\sum_j|\beta_{z,j}|=\infty$.
\end{proposition}
\begin{proof}
Fix $x\in\ZZ^2$ and $m\ge0$, condition on $\FF_m$, and suppose $S_m=y\ne x$. We restart the time index from $m$. If $\ell_z$ cookies have already been used at $z$, the remaining
signed partial sums are $C_z(\ell_z+k)-C_z(\ell_z)$. Only finitely
many $\ell_z$ are nonzero, so
$K_{\rm hist}=\sum_{z:\ell_z>0}|C_z(\ell_z)d(z-x)|$ is finite.
Stop at the first visit to $x$ or exit from $B(x,R)$, where
$R>|y-x|$, and denote this time by $\sigma_R^{(x)}$. It is integrable,
since $|S_n-x|^2-n$ gives an expected exit time of at most $(R+1)^2$.
Let $N_R(z)$ count departures from $z$ during this stopped
segment. The walk does not visit $x$ before $\sigma_R^{(x)}$, so the
indicator term in \eqref{EcondincaSz} vanishes. Grouping the
remaining drift terms by their departure sites gives
\[
 \sum_{k<\sigma_R^{(x)}}b_kd(S_k-x)
 =\sum_{z:N_R(z)>0}\bigl(C_z(\ell_z+N_R(z))-C_z(\ell_z)\bigr)d(z-x).
\]
The subtracted term $-\sum_{z:N_R(z)>0} C_z(\ell_z)d(z-x)$ is at most $K_{\rm hist}$. Thus the
potential-kernel drift identity \eqref{EcondincaSz} and \eqref{eq:disk-rows} give
\[
 \begin{aligned}
 \EE_y a(S_{\sigma_R^{(x)}}-x)
  &\le a(y-x)+u\sum_{|z-x|<R}d_+(z-x)+v\sum_{|z-x|<R}d_-(z-x)+K_{\rm hist}\\
 &\le a(y-x)+\frac{2(u+v)}{\pi}\log R+C+K_{\rm hist}.
 \end{aligned}
\]
On exit, $a(S_{\sigma_R^{(x)}}-x)\ge(2/\pi)\log R-C$, whereas on
hitting $x$ it vanishes. Hence
$$
 \EE_y a(S_{\sigma_R^{(x)}}-x)
 \ge\frac2\pi\log R\cdot\PP_y(\sigma_R^{(x)}\ \text{is an exit from }B(x,R))-C.
$$
Letting $R\to\infty$ shows that, after every finite history $\FF_m$, the conditional probability of never
hitting $x$ is at most $u+v<1$. For fixed $N$ and $m\ge N$, write $A_N^{(x)}:=\{ S_n\ne x \text{ for all } n\geq N\}$.
The conditional probability of $A_N^{(x)}$ given $\FF_m$ is consequently at most $u+v$. As $m\to\infty$
these probabilities converge a.s. to the indicator of that event, so $\PP(A_N^{(x)})=0$. Thus, 
\[
\PP(\exists x \text{ finitely visited}) = \PP(\bigcup_x \bigcup_N A_N^{(x)})=0,
\]
which proves the assertion.
\end{proof}

For example, the stack
$(\theta,-\theta,\theta,-\theta,\ldots)$, with $0<\theta<1$,
has signed partial sums in $[0,\theta]$ and infinite total
absolute excitation. The interval in the proposition must be
common to all vertices.

\begin{proposition}\label{prop:sublattice-cookies}
Let $k_0\ge1$ be an integer. Suppose all cookies are nonnegative,
vanish outside $k_0\ZZ^2$, and have total strength at most $K<k_0^2$
at each vertex. Then the walk almost surely visits every vertex
infinitely often.
\end{proposition}

\begin{proof}
Grouping lattice points into $k_0\times k_0$ cells gives, for all large $R$, uniformly in $x\in\ZZ^2$, 
\[
 \sum_{\substack{z\in k_0\ZZ^2\\|z-x|<R}}d_+(z-x)
 =\frac{2}{\pi k_0^2}\log R+O_{k_0}(1).
\]
Indeed, each complete $k_0\times k_0$ cell contains exactly one
point of $k_0\ZZ^2$ and $k_0^2$ points of $\ZZ^2$. The positive part
of the leading term in \eqref{eq:d-asymptotic} is locally Lipschitz with bound $C|z|^{-3}$  (as in \eqref{eq:disk-rows}'s proof), hence varies by
$O_{k_0}(r^{-3})$ on a cell whose distance from $x$ is $r$, so its value at the $k_0\ZZ^2$ point differs by $O_{k_0}(r^{-3})$ from its average over
the $k_0^2$ lattice points of the cell. Summing over cells thus
replaces $\sum_{z\in k_0\ZZ^2}d_+(z-x)$ by
$\frac1{k_0^2}\sum_{z\in\ZZ^2,\,|z-x|<R}d_+(z-x)$ up to summable
errors, as is the remainder in \eqref{eq:d-asymptotic}.
The $O_{k_0}(R)$ cells
meeting the boundary contribute $O_{k_0}(R^{-1})$ in total, and
the finitely many cells near $x$ contribute a bounded amount.
Comparison with \eqref{eq:disk-rows} proves the claim. 
Applying the drift identity of the preceding proof with
$0\le C_z(\cdot)\le K$ and support in $k_0\ZZ^2$, the spatial cookie bound gives
\[
 \EE_y a(S_{\sigma_R^{(x)}}-x)
 \le a(y-x)+K\sum_{\substack{z\in k_0\ZZ^2\\|z-x|<R}}d_+(z-x)
 \le a(y-x)+\frac{2K}{\pi k_0^2}\log R+C,
\]
where $\sigma_R^{(x)}$ is the first time when $S$ visits $x$ or exit from $B(x,R)$ as in the proof of Proposition \ref{prop:bounded-prefix}. Deleting used cookies preserves these bounds as all cookies are nonnegative. Thus, after every finite history, the probability of never hitting $x$ is at most
$K/k_0^2<1$. The same conditional-probability argument proves the
assertion.
\end{proof}

In particular, one may place three strength-one cookies at every
vertex of $2\ZZ^2$, and no cookies elsewhere.

\section{Acknowledgments}

The author would like to thank Yuval Peres and Alejandro Ramírez for helpful comments on an earlier version of the manuscript. The author used GPT-5.6 to assist with the calculations in Lemma \ref{lem:loglog}. These calculations were independently verified by the author.
 The author is supported by the China Postdoctoral Science Foundation under Grant Numbers 2025M773086 and 2026T190815, and National Natural Science Foundation of China under Grant Number 12271284. 

\bibliographystyle{plain}
\bibliography{M212}

@inproceedings {MR400380,
    AUTHOR = {Burkholder, D. L. and Davis, B. J. and Gundy, R. F.},
     TITLE = {Integral inequalities for convex functions of operators on
              martingales},
 BOOKTITLE = {Proceedings of the {S}ixth {B}erkeley {S}ymposium on
              {M}athematical {S}tatistics and {P}robability ({U}niv.
              {C}alifornia, {B}erkeley, {C}alif., 1970/1971), {V}ol. {II}:
              {P}robability theory},
     PAGES = {223--240},
 PUBLISHER = {Univ. California Press, Berkeley, CA},
      YEAR = {1972},
   MRCLASS = {60G15},
  MRNUMBER = {400380},
MRREVIEWER = {Maurizio\ Pratelli},
}

@article {MR3097419,
    AUTHOR = {Kosygina, Elena and Zerner, Martin P. W.},
     TITLE = {Excited random walks: results, methods, open problems},
   JOURNAL = {Bull. Inst. Math. Acad. Sin. (N.S.)},
  FJOURNAL = {Bulletin of the Institute of Mathematics. Academia Sinica. New
              Series},
    VOLUME = {8},
      YEAR = {2013},
    NUMBER = {1},
     PAGES = {105--157},
      ISSN = {2304-7909,2304-7895},
   MRCLASS = {60K35 (60J80 60K37)},
  MRNUMBER = {3097419},
MRREVIEWER = {Andrew\ R.\ Wade},
}

@article {MR4575009,
    AUTHOR = {Chan, Swee Hong},
     TITLE = {Recurrence of horizontal-vertical walks},
   JOURNAL = {Ann. Inst. Henri Poincar\'e{} Probab. Stat.},
  FJOURNAL = {Annales de l'Institut Henri Poincar\'e{} Probabilit\'es et
              Statistiques},
    VOLUME = {59},
      YEAR = {2023},
    NUMBER = {2},
     PAGES = {578--605},
      ISSN = {0246-0203,1778-7017},
   MRCLASS = {60K35 (60F20 60J10 82C41)},
  MRNUMBER = {4575009},
       DOI = {10.1214/22-aihp1277},
       URL = {https://doi.org/10.1214/22-aihp1277},
}

@inproceedings {MR0047272,
    AUTHOR = {Dvoretzky, A. and Erd\"{o}s, P.},
     TITLE = {Some problems on random walk in space},
 BOOKTITLE = {Proceedings of the {S}econd {B}erkeley {S}ymposium on {M}athematical {S}tatistics and {P}robability},
     PAGES = {353--367},
 PUBLISHER = {Univ. California Press, Berkeley-Los Angeles, Calif.},
      YEAR = {1951},
   MRCLASS = {60.0X},
  MRNUMBER = {47272},
MRREVIEWER = {S.\ Kakutani},
}

@incollection {MR4237270,
    AUTHOR = {Camarena, Daniel and Panizo, Gonzalo and Ram\'{\i}rez,
              Alejandro F.},
     TITLE = {An overview of the balanced excited random walk},
 BOOKTITLE = {In and out of equilibrium 3. {C}elebrating {V}ladas
              {S}idoravicius},
    SERIES = {Progress in Probability},
    VOLUME = {77},
     PAGES = {207--217},
 PUBLISHER = {Birkh\"{a}user, Cham},
      YEAR = {2021},
      ISBN = {978-3-030-60754-8; 978-3-030-60753-1},
   MRCLASS = {60G50 (60G42 82C41)},
  MRNUMBER = {4237270},
       DOI = {10.1007/978-3-030-60754-8\_10},
       URL = {https://doi.org/10.1007/978-3-030-60754-8_10},
}

@book {MR5088005,
    AUTHOR = {Roch, S\'ebastien},
     TITLE = {Modern discrete probability---an essential toolkit},
    SERIES = {Cambridge Series in Statistical and Probabilistic Mathematics},
    VOLUME = {55},
 PUBLISHER = {Cambridge University Press, Cambridge},
      YEAR = {2024},
     PAGES = {xvi+434},
      ISBN = {978-1-00-930511-2; [9781009305129]},
   MRCLASS = {60-01 (05Cxx)},
  MRNUMBER = {5088005},
}

@article {MR1987097,
    AUTHOR = {Benjamini, Itai and Wilson, David B.},
     TITLE = {Excited random walk},
   JOURNAL = {Electron. Comm. Probab.},
  FJOURNAL = {Electronic Communications in Probability},
    VOLUME = {8},
      YEAR = {2003},
     PAGES = {86--92},
      ISSN = {1083-589X},
   MRCLASS = {60G50 (60K37)},
  MRNUMBER = {1987097},
MRREVIEWER = {Michael\ Voit},
       DOI = {10.1214/ECP.v8-1072},
       URL = {https://doi.org/10.1214/ECP.v8-1072},
}

@Article{MR2788390,
  author     = {Benjamini, Ita\i{} and Kozma, Gady and Schapira, Bruno},
  journal    = {C. R. Math. Acad. Sci. Paris},
  title      = {A balanced excited random walk},
  year       = {2011},
  issn       = {1631-073X,1778-3569},
  number     = {7-8},
  pages      = {459--462},
  volume     = {349},
  doi        = {10.1016/j.crma.2011.02.018},
  fjournal   = {Comptes Rendus Math\'{e}matique. Acad\'{e}mie des Sciences.
              Paris},
  mrclass    = {60G50},
  mrnumber   = {2788390},
  mrreviewer = {David\ J.\ Grabiner},
  url        = {https://doi.org/10.1016/j.crma.2011.02.018},
}

@article {MR2041826,
    AUTHOR = {Kozma, Gady and Schreiber, Ehud},
     TITLE = {An asymptotic expansion for the discrete harmonic potential},
   JOURNAL = {Electron. J. Probab.},
  FJOURNAL = {Electronic Journal of Probability},
    VOLUME = {9},
      YEAR = {2004},
     PAGES = {no. 1, 1--17},
      ISSN = {1083-6489},
   MRCLASS = {60J45 (31C20 35C20 60G50)},
  MRNUMBER = {2041826},
MRREVIEWER = {Klaus\ G\"urlebeck},
       DOI = {10.1214/EJP.v9-170},
       URL = {https://doi.org/10.1214/EJP.v9-170},
}

@book {MR2677157,
    AUTHOR = {Lawler, Gregory F. and Limic, Vlada},
     TITLE = {Random walk: a modern introduction},
    SERIES = {Cambridge Studies in Advanced Mathematics},
    VOLUME = {123},
 PUBLISHER = {Cambridge University Press, Cambridge},
      YEAR = {2010},
     PAGES = {xii+364},
      ISBN = {978-0-521-51918-2},
   MRCLASS = {60G50 (60-02)},
  MRNUMBER = {2677157},
MRREVIEWER = {Andrew\ R.\ Wade},
       DOI = {10.1017/CBO9780511750854},
       URL = {https://doi.org/10.1017/CBO9780511750854},
}

@article {MR3531697,
    AUTHOR = {Peres, Yuval and Schapira, Bruno and Sousi, Perla},
     TITLE = {Martingale defocusing and transience of a self-interacting
              random walk},
   JOURNAL = {Ann. Inst. Henri Poincar\'{e} Probab. Stat.},
  FJOURNAL = {Annales de l'Institut Henri Poincar\'{e} Probabilit\'{e}s et
              Statistiques},
    VOLUME = {52},
      YEAR = {2016},
    NUMBER = {3},
     PAGES = {1009--1022},
      ISSN = {0246-0203,1778-7017},
   MRCLASS = {60K35},
  MRNUMBER = {3531697},
MRREVIEWER = {Christian\ Hirsch},
       DOI = {10.1214/14-AIHP667},
       URL = {https://doi.org/10.1214/14-AIHP667},
}

@Article{MR4597323,
  author     = {Angel, Omer and Holmes, Mark and Ramirez, Alejandro},
  journal    = {Ann. Probab.},
  title      = {Balanced excited random walk in two dimensions},
  year       = {2023},
  issn       = {0091-1798,2168-894X},
  number     = {4},
  pages      = {1421--1448},
  volume     = {51},
  doi        = {10.1214/23-aop1622},
  fjournal   = {The Annals of Probability},
  mrclass    = {60K35 (60G42)},
  mrnumber   = {4597323},
  url        = {https://doi.org/10.1214/23-aop1622},
}

\end{document}